\documentclass[11pt,reqno]{article}
\usepackage[utf8]{inputenc}
\usepackage[pagewise]{lineno}
\usepackage[english]{varioref}
\usepackage{dcolumn}
\usepackage[english]{babel}
\usepackage{enumerate}
\usepackage{graphicx}
\usepackage{amsmath}
\usepackage{amsfonts}
\usepackage{amssymb}
\usepackage{amsthm}
\usepackage{mathtools}
\usepackage{caption}
\newtheorem{theorem}{Theorem}

\theoremstyle{definition}
\newtheorem{defn}{Definition}[section]
\usepackage[a4paper,margin=1.0in]{geometry}
\usepackage{tikz}
\usepackage{cleveref}
\crefname{section}{§}{§§}
\Crefname{section}{§}{§§}

{
{
{
{

\newtheorem{thm}{Theorem}[section]

\newtheorem{lem}[thm]{Lemma}
\newtheorem{propos}[thm]{Proposition}

\parindent 0pt
\newcommand{\enter}{\bigskip}

\begin{document}
\thispagestyle{empty} 

\author{Sonali Kaushik\footnote{${}$ Corresponding Author. Email Address: p20180023@pilani.bits-pilani.ac.in} and Rajesh Kumar\\
{Department of Mathematics, BITS Pilani}, \\
{Pilani Campus, Rajasthan-333031, India.}}

\title{Existence and density conservation using a non-conservative approximation for Safronov-Dubovski aggregation equation}
\maketitle

\begin{quote}
{\textit{ \textbf{Abstract.}
The paper deals with the global existence and density conservation for the Safronov-Dubovski equation for three different coefficients $\phi$ such that $\phi_{i,j} \leq \frac{(i+j)}{\min\{i,j\}}$, $\phi_{i,j} \leq (i+j)$ and $\phi_{i,j} \leq (1+i+j)^{\alpha}$ $\forall$ $i,j \in \mathbb{N}$, $\alpha \in [0,1]$. The non-conservative approximation is applied to study the problem and results such as Helly's selection theorem and the refined version of the De la Vall\'ee-Poussin theorem are implemented to establish the existence of each case of the kernel. The article also focuses on the conditions for the conservation of mass per unit volume for such an equation.} 
}

\end{quote}

\section{Prologue}
The Oort-Hulst-Safronov (OHS) equation \cite{oort1946gas, safronov1972evolution} is given by the following expression
\begin{equation}\label{contohs}
\partial_t c(t, x)=-\partial_x\left(c(t, x) \int_{0}^{x} y \phi(x, y) c(t, y)dy\right)-c(t, x) \int_{x}^{\infty} \phi(x, y) c(t, y)dy,
\end{equation}	
for $x,t \in \mathbb{R}^{+}:=[0,+\infty)$. The function $c(t,x)$ is the concentration of particles formed as a result of mutual interaction between particles of different masses (volumes). The term $\phi(x,y)$ is the coalescence coefficient and it represents the frequency of collision between particles of masses $x$ and $y$.  The detail explanation of the terms in the equation (\ref{contohs}) can be found in \cite{safronov1972evolution, lachowicz2003oort}. The Safronov-Dubovski equation was first analyzed in the article  \cite{dubovski1999atriangle} by Dubovski and then studied by several researchers namely Bagland \cite{bagland2005convergence}, Dubovski \cite{dubovski1999atriangle, dubovski1999structural, dubovski2000new}, Davidson \cite{davidson2014existence, davidson2016mathematical}, Kaushik et al. \cite{kaushik2022global,kaushik2022existence, kaushik2022theoretical}. The equation is expressed as
\begin{equation}\label{sd}
\frac{dc_i(t)}{dt}=c_{i-1}(t) \sum_{j=1}^{i-1} j \phi_{i-1,j} c_j(t) -c_i(t) \sum_{j=1}^{i} j \phi_{i,j} c_j(t) -\sum_{j=i}^{\infty} \phi_{i,j}c_i(t)c_j(t),
\end{equation} 
when $t \geq 0$ and the concentration values at $t=0$ is taken as
\begin{equation}\label{sdic}
c_i(0)=c_i^{0} \geq 0.
\end{equation}
This equation describes the collision of particles of mass $i s_0$ and $j s_0$, with $s_0>0$ being the mass of the smallest particle in the system. This collision creates monomers from the smaller cluster that leads to the sticking of the monomers with the larger cluster. The first summation represents the coagulation of an $(i-1)$-mer and a monomer and creates a $i$-sized particle in the dynamics of the system. Followed by this, the two terms denote the destruction of an $i$-mer. The penultimate term is symbolic of the formation of particles of mass larger than $i s_0$. The last expression takes its place in the equation when a $i$-sized clump dissolves into dust and merge with a $j-$ sized particle. The variables $c_i(t)$ denote the concentration of particles of mass $is_0$ at certain point of time $t$ and $\phi_{i,j}$, $i \neq j$ defines the rate at which bundles of size $i$ and $j$ collide and is called the coagulation kernel. 
 The kernel $\phi_{i,j} \geq 0$ is symmetric, i.e, $\phi_{i,j}=\phi_{j,i}$. \enter

The moments of the solution are of utmost importance in dealing with the boundedness of important quantitites. The $r^{\text{th}}$ moment has been defined in the literature as
\begin{equation}\label{rthmoment}
\mu_r(t) = \sum_{i=1}^{\infty} i^r c_i(t).
\end{equation}
If we take $r=0$, the quantity $\mu_0(t)$ defines the total number of particles and $r=1$ gives us the total mass of the particles. Putting $r=2$ defines the second moment of the solution which can be interpreted as the energy dissipated by the system \cite{dubovskiui1994mathematical}. \\


\subsection{Existing Literature}

Dubovski \cite{dubovski1999structural}, in 1999, developed the model (\ref{sd}) and demonstrated preliminary results for the general kernel $\phi_{i,j}$ including the non-negativity and mass conservation law. The author also discussed the significance of the Oort-Hulst-Safronov model in calculating coagulation front velocity (the rate of displacement of the boundary of the non-zero values of the distribution function). Dubovski also emphasized the relationship between the coagulation front velocity and the mass conservation law, stating that if the velocity escapes to infinity, the equation does not obey this law.
Furthermore, Lachowicz et al. \cite{lachowicz2003oort} pointed out differences between the OHS model (\ref{contohs}) and continuous version of Smoluchowski \cite{smoluchowski1917mathematical} equation and with the help of an $\epsilon$ dependent family of coagulation equations (where $\epsilon \in (0,1]$), the former is shown to lead as $\epsilon \to 0$ and when $\epsilon=1$, it is noticed to correspond to the latter.

 The article \cite{laurenccot2005convergence} studied the self-similar profiles by finding Lyapunov functions for the constant parameter $\phi \equiv 1$. In \cite{bagland2007self}, the authors established the existence for the kernel
$\phi(x,y)=x^\lambda+y^{\lambda},\hspace{.1 cm} 0 < \lambda <1 .$
The author \cite{LAURENCOT200680} analyzed the OHS model \eqref{contohs} and studied the conditions for \textit{collapse} \cite{dubovski1999structural} when the parameter in consideration is $\phi(x,y)=xy$. An article by Barik et al. \cite{barik2022mass} dealt with the global existence for the continuous OHS equation (\ref{contohs}) for a singular paramter.
The article \cite{bagland2005convergence} discussed the existence of solution for the equation (\ref{sd}) in case of the kernel of the form, $$\lim_{l \to \infty} \frac{\phi_{i,l}}{l}=0.$$
In 2014, Davidson \cite{davidson2014existence} proved the global existence of the solution for the bounded coagulation parameter of the form $j \phi_{i,j} \leq M$ where $M$ is a constant and in addition to existence, the density is also shown to be conserved for unbounded kernels $\phi_{i,j}\leq C_{\phi} h_i h_j$ where $\frac{h_i}{i} \to 0$ as $i \to \infty$. His work includes uniqueness but only for the bounded kernel of the form $\phi_{i,j} \leq C_{\phi}$.
Recently, in 2022, Kaushik et al. \cite{kaushik2022existence, kaushik2022theoretical} proved the existence, density conservation and uniqueness for the kernel of the form $\phi_{i,j} \leq \frac{(i+j)}{\min\{i,j\}}$ and discussed the differentiability, existence and density conservation for the kernel, $\phi_{i,j} \leq (i+j)$. However, the existence and mass conservation of the equation is not yet analyzed in case of the non-conservative approximation [see \cite{barik2020mass} for the coagulation-fragmentation equation].

\medskip

So, this article aims to prove the global existence result and establish the density conservation for the kernels $\phi_{i,j} \leq \frac{(i+j)}{\min\{i,j\}}$, $\phi_{i,j} \leq (i+j)$ and $\phi_{i,j} \leq (1+i+j)^{\alpha}$ $\forall$ $i,j \in \mathbb{N}$, $\alpha \in [0,1]$ when (\ref{sd}) is approximated using non-conservative form. 
The existence of the solution is proved by assuming that $\mu_{1,n}(0)$ is finite and $\mu_{1,n}(t)$ is bounded.
To prove the density conservation, an additional physical property of the kernel is required, i.e.,
 \begin{equation}
 \label{compactsupp}
 \phi_{n,j} \to 0 \hspace{.2 cm} \text{as} \hspace{.2 cm} n \to \infty \hspace{.2 cm}  \text{for every} \hspace{.2 cm} j  \in [1,(n-1)],
 \end{equation} 
%
  
 The novelty of this work is that the existence, density conservation of the solution for the discrete Safronov-Dubovski equation is established by approximating the equation using the non-mass conserving truncated system. 
 The refined version of De la Vall\'ee-Poussin theorem is used to prove the global existence and some properties of the convex function ensure the conservation of the first moment, i.e., mass.
\medskip

The definitions that are essential to lay the groundwork for the results are mentioned. 
The set of finite mass sequences is defined by 
\begin{equation}\label{spaceX}
B=\{z=(z_{d}): ||z|| <\infty\},
\end{equation}
with 
\begin{equation}\label{norm1}
||z||:=\sum_{d=1}^{\infty} d |z_{d}|,
\end{equation}
where $(B, \|\cdot\|)$ is a Banach space. Since, only the non-negative solutions of \eqref{sd}-\eqref{sdic} are of interest, we consider the non-negative cone
\begin{equation}\label{coneX+}
B^{+}= \{c=(c_{i}) \in B: c_{i} \geqslant 0\}.
\end{equation}
The concentration $c_i(t)$ is defined below as 
\begin{defn} \label{defsol}
Let $T \in (0,\infty].$ A solution $c=(c_i)$ of the initial value problem \eqref{sd}--\eqref{sdic} is a function 
$c:[0,T) \to B^{+}$ such that 
\begin{enumerate}
\item $\forall \hspace{.1cm} i, \hspace{.1cm} c_i$ is continuous and $\sup_{0\leq t < T} ||c|| <\infty$,
\item $\int_{0}^{t} \sum_{j=1}^{\infty} \phi_{i,j}c_j(h) dh <\infty$ for every $i$ and $\forall \hspace{.2cm} 0 \leq t <T$,
\item
\begin{equation}\label{solution}
c_i(t)=c_{i}(0)+\int_{0}^{t}\Bigl(\delta_{i\geqslant 2} c_{i-1}(h)\sum_{j=1}^{i-1} j \phi_{i-1,j} c_j(h)-c_i(h)\sum_{j=1}^{i} j \phi_{i,j} c_j(h)-
c_i(h)\sum_{j=i}^{\infty} \phi_{i,j}c_j(h)\Bigr) dh,
\end{equation}
\end{enumerate}
where $\delta_P=1$ if $P$ is true, and is zero otherwise. 
\end{defn}

The paper is organized as follows: In Section 2, finite-dimensional non-conservative system for the equation (\ref{sd}) and the required estimates on the properties of the system which shall be used in Section 3 that presents the global existence theorem are also presented.
The proof of the conservation of the first moment for equation (\ref{sd}) is discussed in Section 4.
\section{A Truncated System and Required Results}\label{fds}
The section addresses a method for dealing with an infinite system of ordinary differential equations comprising of infinite summations and analyzes the properties of the said system after its conversion to a finite system of equations. The estimates obtained for the parameters involved in the finite dimensional equation aids in passing the limit $n \to \infty$ which enables in proving important results for the concerned model (\ref{sd})-(\ref{sdic}).

We present an approximating system of equations which do not conserve the first moment. The system is hence called \emph{non-conservative} and is defined as
\begin{equation}\label{truncation2}
\frac{dc^n_i}{dt}= C^{n}_i(c), 
\qquad \text{for $1 \leq i \leq n$,}
\end{equation}
where
\begin{align}
C^{n}_1(c) &:= - \phi_{1,1}c_1^2-c_1\sum_{j=1}^{n}\phi_{1,j}c_j  \label{1aa'} \\
C^{n}_i(c) &:=c_{i-1}\sum_{j=1}^{i-1}j\phi_{i-1,j}c_j-c_i\sum_{j=1}^{i} j\phi_{i,j}c_j-c_i\sum_{j=i}^{n}\phi_{i,j}c_j , 
\quad \text{for $2 \leq i \leq n-1$,} \label{1a'}  \\
C^{n}_n(c) &:=c_{n-1}\sum_{j=1}^{n-1}j\phi_{n-1,j}c_j\label{1b'}
\end{align}
with the initial conditions as 
\begin{equation}\label{tic}
c_i(0)  = c_i^{0} \geq 0, \qquad\text{for $1\leq i \leq n$}.
\end{equation} We compute the $i^{th}$ moment of the solution of the finite dimensional system which aids in computing its zeroth and first moments.
\begin{lem}\label{lemma1}
Let $c=(c_i)_{i\in\{1, \ldots, n\}}$ be a solution of (\ref{truncation2})-(\ref{1b'})
defined in an open interval $I$ containing $0$. Let $w=(w_i)$ be a real sequence. Then
\begin{equation}
\frac{d\mu_w^n}{dt} = \sum_{i=1}^{n-1}|w_{i+1}-w_{i}|\sum_{j=1}^{i}j\phi_{i,j}c_ic_j-\sum_{i=1}^{n-1}\sum_{j=i}^{n}w_i\phi_{i,j}c_ic_{j}\label{momenteq2}.
\end{equation}
\end{lem}

\begin{proof}
Using the expressions \eqref{1aa'}--\eqref{1b'}, one can obtain
\begin{align*} 
\frac{d\mu_w^n}{dt} = \sum_{i=1}^{n} w_i C^{n}_i(c)
                                   & =  \underbrace{w_1\Bigl(-\phi_{1,1}c_1^2-c_1\sum_{j=1}^{n}\phi_{1,j}c_j\Bigr)}_{I} \, + \underbrace{w_n\Bigl(c_{n-1}\sum_{j=1}^{n-1}j\phi_{n-1,j}c_j\Bigr)}_{II}\\
                                    & \;\;\;\; +  \underbrace{\sum_{i=2}^{n-1} w_i\Bigl(c_{i-1}\sum_{j=1}^{i-1}j\phi_{i-1,j}c_j-c_i\sum_{j=1}^{i} j\phi_{i,j}c_j-c_i\sum_{j=i}^{n}\phi_{i,j}c_j\Bigr)}_{III}.
\end{align*}
To obtain the desired result, some simple but careful computations are required. One can proceed as follows: expand $III$ for each $i$, then combine the first term in $I$ and the first term in $III_1$, where $III_k$, $k=1,2 \cdots (n-1)$ denotes sub-brackets of $III$. Further, merge the second term of $III_1$ with the first term of $III_2$ and the second term of $III_2$ with the first term of $III_3$. Finally, combine $II$ and the second term of $III_{(n-1)}$.
\end{proof}
The validity of this truncation for a coagulation system of equation can be established by computing the rate of change in the number of particles of the finite dimensional system. To do so, putting $w_i=1$ in \eqref{momenteq2} and the non-negativity of $c_i,c_{j}$ and $\phi_{i,j}$ yield
\begin{equation*}
\frac{d\mu^n_0}{dt}=\sum_{i=1}^{n}C^{n}_{i}=-\sum_{i=1}^{n-1}\sum_{j=i}^{n}\phi_{i,j}c_ic_{j} \leq 0
\end{equation*} 
and thus $\mu_{0}^{n}(t) \leq \mu_{0}^{n}(0) \hspace{.2 cm} \forall t \in I \cap \{t \geq 0\}$
which ensures that the above mentioned truncation is valid.\medskip

Taking $w_1=i$ in \eqref{momenteq2}, we obtain 
\begin{equation*}\frac{d\mu_1^{n}(t)}{dt}=\sum_{i=1}^{n} i C^n_{i}=\sum_{i=1}^{n-1}\sum_{j=1}^{i} j\phi_{i,j}c_{i,j}-\sum_{i=1}^{n-1}\sum_{j=i}^{n}i\phi_{i,j}c_{i}c_{j}.
\end{equation*}
Now, adding and subtracting some terms, changing the order of summation and replacing $i \leftrightarrow j$ in the second term, it is easy to see that
\begin{equation}\label{truncated1moment}
\frac{d\mu_1^{n}(t)}{dt}=- \sum_{j=1}^{n-1}j \phi_{n,j}c_{n}c_{j} \leq 0,
\end{equation}
which means that solutions to the truncated system do not conserve mass, i.e.,
\begin{equation}\label{truncatedmassnonconservation}
\mu^{n}_1(t) \leq \mu^{n}_1(0), \quad \text{for} \hspace{.1 cm} \text{any} \hspace{.2 cm} t.
\end{equation}
When the approximating system does not conserve the first moment, the global existence of the solution is established 
by the application of Helly's theorem (see \cite{kolmogorov1975introductory}) to $\{c^{n}_{i}(t)\}$. The theorem requires us to prove that the truncated solution is of locally bounded total variation and uniformly bounded at a point. This is achieved by the application of Lemma \ref{second}, discussed later on. In addition to this, we make use of the refined version of De la Vall\'ee- Poussin theorem (see \cite{filbet2004mass, dellacherie1975probabilites}), which ensures that if $c^{0}_{i}$ is in weighted $L^1$ space, there exists a non-negative convex function $\gamma \in G$ where
\begin{equation*}
G :=\{\gamma \in C^1([0,\infty))\cap W^{1,\infty}_{loc}(0,\infty): \gamma(0)=0,\gamma'(0) \geq 0, \gamma' \hspace{.15 cm} \text{is \hspace{.02 cm} concave \hspace{.02 cm} and} \hspace{.15 cm} \lim_{r \to \infty} \frac{\gamma(r)}{r}= \infty\},
\end{equation*}
satisfying certain properties given by the following proposition.

\begin{propos}\label{prop1}
Let, $i,j \geq 1$ and $\gamma \in G$, then the following holds true $$0 \leq \gamma(i+1)-\gamma(i) \leq \frac{(3i+1)\gamma(1)+2 \gamma(i)}{(i+1)}.$$
\end{propos}
\begin{proof}
The proof can be adapted from Lemma A.2 in \cite{laurenccot2001lifshitz}.
\end{proof}
Now, the results which will be useful in proving the existence are presented below.
\begin{lem}\label{first}
Consider $T \in (0,\infty)$, $t \in [0,T]$, and $\gamma \in G$, then  for every $n\geq 2$ and $\phi_{i,j} \leq \frac{(i+j)}{\min\{i,j\}}$, $\phi_{i,j} \leq (i+j)$ and $\phi_{i,j} \leq (1+i+j)^{\alpha}$ $\forall$ $i,j \in \mathbb{N}$, $\alpha \in [0,1]$, the following holds
\begin{equation}
\label{qt}
\sum_{i=1}^{n} \gamma(i)c_{i}^{n}(t) \leq Q_b(T) \hspace{.5 cm} \text{and} 
\end{equation}
\begin{equation}
\label{qtt}
0 \leq \int_{0}^{t}\left(\sum_{i=1}^{n-1}\sum_{j=i}^{n} \gamma(i) \phi_{i,j}c^{n}_{i}(h)c^{n}_{j}(h)\right)dh\ \leq Q_b(T),
\end{equation}
where $Q_b(T)$ depends only on $\gamma(1), ||c_{0}||$ and $\mu_0(0)$ for $b=1,2,3$.
\end{lem}
\begin{proof}
Using $w_{i}=\gamma(i)$ in the equation (\ref{momenteq2}) leads to 
\begin{equation}\label{l}
\sum_{i=1}^{n} \gamma(i)c_{i}^{n}(t)=\sum_{i=1}^{n}\gamma(i)c^{0}_{i}+\int_{0}^{t}\Big(\sum_{i=1}^{n-1}|\gamma(i+1)-\gamma(i)|\sum_{j=1}^{i}j\phi_{i,j} c_{i}^{n}(h)c_{j}^{n}(h)-\sum_{i=1}^{n-1}\sum_{j=i}^{n}\gamma(i)\phi_{i,j}c_{i}^{n}(h)c_{j}^{n}(h)\Big)dh.
\end{equation}
Since $\gamma(k)$ and $c_{i}^{n}$ for $k=i,j$ are non-negative, the above expression simplifies to
\begin{equation*}
\sum_{i=1}^{n} \gamma(i)c_{i}^{n}(t) \leq \sum_{i=1}^{n}\gamma(i)c^{0}_{i}+ \int_{0}^{t}\Big(\sum_{i=1}^{n-1}|\gamma(i+1)-\gamma(i)|\sum_{j=1}^{i}j\phi_{i,j} c_{i}^{n}(h)c_{j}^{n}(h)\Big)dh.
\end{equation*}
Using Proposition \ref{prop1}, replacing $\phi_{i,j} \leq \frac{(i+j)}{j}$, $\phi_{i,j} \leq (i+j)$ and $\phi_{i,j} \leq (1+i+j)^{\alpha}$ respectively will give us the following expressions
\begin{align*}
\sum_{i=1}^{n} \gamma(i)c_{i}^{n}(t)
&\leq \sum_{i=1}^{n}\gamma(i)c^{0}_{i}+\int_{0}^{t}\Big(\sum_{i=1}^{n-1}\sum_{j=1}^{i}[(3i+1)\gamma(1)+2\gamma(i)] 2 c_{i}^{n}(h)c_{j}^{n}(h)\Big)dh\\
&\leq \sum_{i=1}^{n}\gamma(i)c^{0}_{i}+\int_{0}^{t}\Big(Q_{1,1}(T)+Q_{1,2}(T)\sum_{i=1}^{n}\gamma(i)c^{n}_{i}(h)\Big)dh,
\end{align*}
\begin{align*}
\sum_{i=1}^{n} \gamma(i)c_{i}^{n}(t)
&\leq \sum_{i=1}^{n}\gamma(i)c^{0}_{i}+\int_{0}^{t}\Big(\sum_{i=1}^{n-1}\sum_{j=1}^{i}[(3i+1)\gamma(1)+2\gamma(i)] 2j c_{i}^{n}(h)c_{j}^{n}(h)\Big)dh\\
&\leq \sum_{i=1}^{n}\gamma(i)c^{0}_{i}+\int_{0}^{t}\Big(Q_{2,1}(T)+Q_{2,2}(T)\sum_{i=1}^{n}\gamma(i)c^{n}_{i}(h)\Big)dh,
\end{align*}
\begin{align*}
\sum_{i=1}^{n} \gamma(i)c_{i}^{n}(t)
&\leq \sum_{i=1}^{n}\gamma(i)c^{0}_{i}+\int_{0}^{t}\Big(\sum_{i=1}^{n-1}\sum_{j=1}^{i}\frac{[(3i+1)\gamma(1)+2\gamma(i)]}{i} j(1+2i)^{\alpha} c_{i}^{n}(h)c_{j}^{n}(h)\Big)dh\\
&\leq \sum_{i=1}^{n}\gamma(i)c^{0}_{i}+\int_{0}^{t}\Big(Q_{3,1}(T)+Q_{3,2}(T)\sum_{i=1}^{n}\gamma(i)c^{n}_{i}(h)\Big)dh,
\end{align*}
where $Q_{1,1}(T)=2 \gamma(1) \mu_0(0)(\mu_0(0)+3\|c_0\|)$, $Q_{1,2}(T)=4 \mu_0(0)$ and $Q_{2,1}(T)=2 \gamma(1) ||c_0||(\mu_0(0)+3\|c_0\|)$, $Q_{2,2}(T)=4 ||c_0||$ and $Q_{3,1}(T)= \gamma(1) ||c_0||(9||c_0||+2\mu_0(0))$, $Q_{3,2}(T)=6 ||c_0||$ respectively.
Finally, an application of Gronwall's inequality proves (\ref{qt}) where $Q_b(T)=\sum_{i=1}^{n}\gamma(i)c^{0}_{i}+\frac{Q_{b,1}(T)}{Q_{b,2}(T)}\left(e^{T Q_{b,2}(T)}-1\right)$. Combining (\ref{qt}) and (\ref{l}) leads to
\begin{equation*}
\int_{0}^{t}\left(\sum_{i=1}^{n-1}\sum_{j=i}^{n} \gamma(i) \phi_{i,j}c^{n}_{i}(h)c^{n}_{j}(h)\right)dh \leq Q_b(T),
\end{equation*}
for every $b=1,2,3$ which establishes (\ref{qtt}).
\end{proof}
\begin{lem}\label{second}
Consider $T \in (0,\infty)$, $i \in \mathbb{N}$ and $\phi_{i,j} \leq \frac{(i+j)}{\min\{i,j\}}$, $\phi_{i,j} \leq (i+j)$ and $\phi_{i,j} \leq (1+i+j)^{\alpha}$ $\forall$ $i,j \in \mathbb{N}$, $\alpha \in [0,1]$. The following result holds true for a constant $\tilde{Q}_b(T)$
\begin{equation}
\label{qt2}
\sup_{0\leq t \leq T}\int_{0}^{t} \Big|\frac{d c^{n}_{i}}{dt}\Big|dh =\tilde{Q}_b(T),
\end{equation}
 which depends on $||c_{0}||, \mu_0(0)$ and $T$ 
\begin{proof}
Using the equations (\ref{truncation2}) and (\ref{1a'}) yield for $\phi_{i,j} \leq \frac{(i+j)}{\min\{i,j\}}$
\begin{align*}
\int_{0}^{t} \Big|\frac{d c^{n}_{i}}{dt}\Big|dh &\leq \int_{0}^{t}\Bigl(\Big|c^{n}_{i-1}(h)\sum_{j=1}^{i-1}j \phi_{i-1,j}c^{n}_{j}(h)\Big|+\Big|c^{n}_{i}(h)\sum_{j=1}^{i}j \phi_{i,j} c^{n}_{j}(h)\Big|+\Big|c^{n}_{i}(h)\sum_{j=i}^{\infty}\phi_{i,j}c^{n}_{j}(h)\Big|\Bigr)dh\\
&\,\,\,\,=\int_{0}^{t} \Bigl(2 \Big|c^{n}_{i}(h)\sum_{j=1}^{i}j \phi_{i,j} c^{n}_{j}(h)\Big|+\Big|c^{n}_{i}(h)\sum_{j=i}^{\infty}\phi_{i,j}c^{n}_{j}(h)\Big|\Bigr)dh\\
&\,\,\,\, \leq 6 \mu_0(0)||c_0|| T,
\end{align*}
and in case of $\phi_{i,j} \leq (i+j)$,
\begin{align*}
\int_{0}^{t} \Big|\frac{d c^{n}_{i}}{dt}\Big|dh \leq  4 ||c_0||^2 T+2 \mu_0(0)||c_0||T.
\end{align*}
Finally for the kernel $\phi_{i,j} \leq (1+i+j)^{\alpha}$, we have
\begin{align*}
\int_{0}^{t} \Big|\frac{d c^{n}_{i}}{dt}\Big|dh \leq 4 \mu_0(0)||c_0||T+ 4 ||c_0||^2 T+\mu_0(0)^2 T.
\end{align*}
Also, the absolute values of the equations (\ref{1aa'}) can easily be shown bounded by $2 ||c_0||^2+\mu_0(0)||c_0||$, $2 ||c_0||^2+\mu_0(0)||c_0||$ and $\frac{5}{2} ||c_0||^2+2 ||c_0||\mu_0(0)$ respectively for the three cases of parameters. In additon, we can show that the absolute values of the equations \eqref{1b'} are bounded by $2||c_{0}||\mu_0(0)$, $2 ||c_0||^2$ and  $2 ||c_0||^2$ respectively,
which completes the proof of the lemma.
\end{proof}
\end{lem}
Now, we can proceed to study the global existence result for (\ref{sd})-(\ref{sdic}) considering the conservative as well as the non-conservative truncation. 
\section{Global Existence Theorem}\label{existence}
This section is devoted to proving the global existence of a solution for the non-conservative truncation.
\begin{theorem}\label{existence_nc}
Consider $c \in B^{+}$ and $\phi_{i,j} \leq \frac{(i+j)}{\min\{i,j\}}$, $\phi_{i,j} \leq (i+j)$ and $\phi_{i,j} \leq (1+i+j)^{\alpha}$ $\forall$ $i,j \in \mathbb{N}$, $\alpha \in [0,1] \hspace{.2 cm} \forall i,j$. Suppose that $c^{n}_{i}(t)$ and $c^{0}_{i}$ are the solution and initial condition of the non-conservative approximation (\ref{truncation2})-(\ref{1b'}), then there exists a solution of the discrete OHS equation (\ref{sd})-(\ref{sdic}) in $\mathbb{R}^{+}$. 
\end{theorem}
\begin{proof}
Let, $(E,\sigma,M)$ be a measure space with $E=\mathbb{N}, \sigma=\{S:S \subset \mathbb{N}\}$ and the measure $M$ defined by 
$$M(S)=\sum_{i \in S} c_{i}^{0}.$$
Since, $c^{0} \in B^{+}$, we get, $z \mapsto z \in L^1(E,\sigma,M)$. Then by the refined version of De la Vall\'ee- Poussin theorem, there exists a function $\gamma_0 \in G$ such that $z \mapsto \gamma_0(z) \in L^1(E,\sigma,M)$ meaning that
\begin{equation}\label{gamma_0}
\sum_{i=1}^{\infty}\gamma_0(i)c_{i}^{0} < \infty.
\end{equation}
Using the equation (\ref{truncatedmassnonconservation}) and Lemma \ref{second}, the sequence is locally bounded and absolutely continuous in $(0,T)$ for every $i \geq 1$ and $T \in (0,\infty)$. Thus, by Helly's theorem, there exists a subsequence of $\{c^{n}_{i}\}$ denoted as $\{\hat{c^{n}_{i}}\}$ and a sequence $\{c_{i}\}$ of functions of locally bounded variation such that 
\begin{equation}\label{limit}
\lim_{n \to \infty} \hat{c^{n}_{i}} = c_i(t), \hspace{.2 cm}\forall i \geq 1, t \geq 0.
\end{equation}
Also, if we let $c_{i}^{n}=0$ when $i >n$,
\begin{equation}\label{non-negativity}
||c^n(t)||=\sum_{i=1}^{n}ic^{n}_{i}(t) \leq \sum_{i=1}^{n}ic^{n}_{0 i}\leq ||c_{0}||
\end{equation}
which, followed by the use of (\ref{limit}) gives the non-negativity of the solution $c_{i}(t)$.
Further, since, $\gamma_{0} \in G$, equation (\ref{gamma_0}) and Lemma \ref{first} yield the following
 \begin{equation}\label{one}
 \sum_{i=1}^{n}\gamma_{0}(i)c^{n}_{i}(t) \leq Q(T) \hspace{.5 cm} \& \hspace{.5 cm} 0 \leq \int_{0}^{T}\sum_{i=1}^{n-1}\sum_{j=i}^{n} \gamma_{0}(i) \phi_{i,j}c^{n}_{i}(h)c^{n}_{j}(h)dh \leq Q(T),
 \end{equation}
 for every $n \geq 2, t \in [0,T]$ where $T \in [0,\infty)$.
 If we take $T \in (0,\infty)$ and $ m \geq 2$, then using the above equation (\ref{one}) and the non-negativity of $\gamma(i)$ give us
 \begin{equation}\label{two}
 \sum_{i=1}^{m}\gamma_{0}(i)c^{n}_{i}(t) \leq Q(T) \hspace{.5 cm} \& \hspace{.5 cm} 0 \leq \int_{0}^{T}\sum_{i=1}^{m-1}\sum_{j=i}^{m} \gamma_{0}(i) \phi_{i,j}c^{n}_{i}(h)c^{n}_{j}(h)dh \leq Q(T),
 \end{equation}
 Using (\ref{limit}), passing the limit as $n \to \infty$ in equation (\ref{one}) and $m \to \infty$ in (\ref{two}) yield
 \begin{equation}\label{three}
 \sum_{i=1}^{\infty}\gamma_{0}(i)c_{i}(t) \leq Q(T) \hspace{.5 cm} \& \hspace{.5 cm} 0 \leq \int_{0}^{T}\sum_{i=1}^{\infty}\sum_{j=i}^{\infty} \gamma_{0}(i) \phi_{i,j}c_{i}(h)c_{j}(h)dh \leq Q(T).
\end{equation}
Following on the lines of \cite{laurenccot2002discrete}, the application of (\ref{non-negativity}), (\ref{three}), definitions of the kernels and the properties of $\gamma_0$ confirm the Definition \ref{defsol}(2) as
\begin{equation}\label{Definition 1.1(2)}
\int_{0}^{T} \sum_{j=1}^{\infty} \phi_{i,j}c_{j}(h)dh <\infty.
\end{equation}
Now, by the definitions of $\phi_{i,j}$ and (\ref{non-negativity}), we have 
\begin{equation*}
\sum_{j=i}^{n}\phi_{i,j}c^{n}_{i}c^{n}_{j} \leq 2 jc^{n}_{i}c^{n}_{j} \leq 2 ||c_{0}||\mu_0(0),
\end{equation*}
for the first two parameters
and \begin{equation*}
\sum_{j=i}^{n}\phi_{i,j}c^{n}_{i}c^{n}_{j} \leq 2 jc^{n}_{i}c^{n}_{j} \leq 2 ||c_{0}||\mu_0(0)+\mu_0(0)^2,
\end{equation*}
for the last one.
Further, using (\ref{limit}) followed by the Lebesgue-dominated convergence theorem, the following holds 
\begin{equation}\label{convergence}
\lim_{n \to \infty} \int_{0}^{T} \Big(\sum_{j=i}^{n}\phi_{i,j}c^{n}_{i}c^{n}_{j}-\sum_{j=i}^{\infty}\phi_{i,j}c_{i}c_{j}\Big)(h)dh=0.
\end{equation}
The relations (\ref{limit}), (\ref{truncatedmassnonconservation}), (\ref{non-negativity}) and (\ref{convergence}) lead to the fulfilment of Definition \ref{defsol}(3). Consequently using (\ref{Definition 1.1(2)}), the Definition \ref{defsol}(1) follows, hence establishing the existence result for the problem (\ref{sd})-(\ref{sdic}) in $\mathbb{R}^{+}$.
\end{proof}

\section{Density Conservation}\label{density} 
The density conservation is demonstrated for the non-mass conserving truncation by the following result.
\begin{theorem}\label{density_nc}
Let $\phi_{i,j}\leqslant i j$ for all natural numbers $i$ and $j$.
Let $c=(c_i)\in B^+$ be a solution of the Safronov-Dubovski equation (\ref{sd}) satisfying all the conditions of Theroem  \ref{existence_nc}. Additionally, if (\ref{compactsupp}) holds, then the solution is density conserving satisfying
\begin{equation}
\label{densityconservation_nc}
\sum_{i=1}^{\infty} i c_{i}(t)= \sum_{i=1}^{\infty} i c_{i}(0). 
\end{equation}
\end{theorem}
\begin{proof}
To establish the density conservation, let us consider
\begin{align}\label{d_nc}
\Big|\sum_{i=1}^{\infty}i c_{i}(t)-\sum_{i=1}^{\infty} i c_{i}(0)\Big| &=\Big|\big(\sum_{i=1}^{m}+\sum_{i=m+1}^{\infty}\big) i c_{i}(t)- \big(\sum_{i=1}^{m}+\sum_{i=m+1}^{\infty}\big) i c_{i}(0)\Big|.
\end{align}
Now, writing the equation (\ref{truncated1moment}) as 
\begin{equation}
\label{d_nc1}
\sum_{i=1}^{n}ic_{i}(0)=\sum_{i=1}^{m}ic^{n}_{i}(t)+\sum_{i=m+1}^{n} i c_{i}^{n}(t)-\sum_{j=1}^{n-1}j \phi_{n,j}c^{n}_{n}(t)c^{n}_{j}(t).
\end{equation}
Using (\ref{d_nc1}) in (\ref{d_nc}), one can obtain
\begin{align}\label{d_nc2}
\nonumber
\Big|\sum_{i=1}^{\infty}i c_{i}(t)-\sum_{i=1}^{\infty} i c_{i}(0)\Big| & \leq \sum_{i=1}^{m}i \Big|c_{i}(t)-c^{n}_{i}(t)\Big|+ \sum_{i=m+1}^{\infty}i c_{i}(t)+\sum_{i=m+1}^{n}i c^{n}_{i}(t)+\sum_{i=n+1}^{\infty} i c_{i}(0)\\
& +\Big|c_{n}^{n}(t) \sum_{j=1}^{n-1} j \phi_{n,j}c_{j}^{n}(t)\Big|.
\end{align}

Before passing the limit $n \to \infty$, the boundedness of the last expression in the above equation
needs to be dealt and is given as for $j\phi_{i,j} \leq (i+j)$
\begin{align*}
\Big|c_{n}^{n}(t) \sum_{j=1}^{n-1} j \phi_{n,j}c_{j}^{n}(t)\Big| & \leq |c_{n}^{n}(t) \sum_{j=1}^{n} j \phi_{n,j}c_{j}^{n}(t)| \leq  2 ||c_0|| \mu_0(0),
\end{align*}
in case of $\phi_{i,j} \leq (i+j)$,
\begin{align*}
\Big|c_{n}^{n}(t) \sum_{j=1}^{n-1} j \phi_{n,j}c_{j}^{n}(t)\Big| \leq  2 ||c_0||^2,
\end{align*}
and finally for the kernel $\phi_{i,j} \leq (1+i+j)^{\alpha}$, we have
\begin{align*}
\Big|c_{n}^{n}(t) \sum_{j=1}^{n-1} j \phi_{n,j}c_{j}^{n}(t)\Big| \leq  2 ||c_0||^2+||c_0|| \mu_0(0).
\end{align*}
 Now, having (\ref{compactsupp}), then passing the limit in (\ref{d_nc2}) and using (\ref{limit}) provide
\begin{align*}
\nonumber
\lim_{n \to \infty} \Big|\sum_{i=1}^{\infty}i c_{i}(t)-\sum_{i=1}^{\infty} i c_{i}(0)\Big| & \leq  \sup_{i \geq m} \frac{2 i c_{i}^{n}(t)\gamma_0(i)}{\gamma_0(i)}.
\end{align*}
Hence, the first result given in equation (\ref{three}) and the properties of $\gamma_0$ lead to accomplish our claim.
\end{proof}

\section*{Acknowledgement}
Rajesh Kumar wishes to thank Science and Engineering Research Board (SERB), Department of Science and Technology (DST), India, for the funding through the project MTR/2021/000866.

\section*{Declarations of Interest}
There are no conflicts of interest to this work.

\bibliography{sonali_final}
\bibliographystyle{ieeetr}
\end{document}